\documentclass{article}

\usepackage[english]{babel}
\usepackage{amsmath,amssymb,amsthm, amsfonts}
\usepackage{verbatim, enumerate}
\usepackage{graphicx}
\usepackage{url, color}
\usepackage[mathlines]{lineno}
\usepackage{dsfont} 

\usepackage[colorlinks=true, allcolors=blue]{hyperref}

\numberwithin{figure}{section}   

\definecolor{red}{rgb}{.8,0,0}
\def\red{\color{red}}
\definecolor{bblu}{rgb}{0,0,1}

\newtheorem{thm}{Theorem}[section]
\newtheorem{cor}[thm]{Corollary}
\newtheorem{lem}[thm]{Lemma}
\newtheorem{prop}[thm]{Proposition}

\newtheorem{obs}[thm]{Observation}

\theoremstyle{definition}
\newtheorem{rem}[thm]{Remark}

\theoremstyle{definition}

\theoremstyle{definition}
\newtheorem{ex}[thm]{Example}

\newcommand{\mtx}[1]{\begin{bmatrix} #1 \end{bmatrix}}
\newcommand{\spec}{\operatorname{spec}}

\newcommand{\rank}{\operatorname{rank}}
\newcommand{\diag}{\operatorname{diag}}
\newcommand{\tr}{\operatorname{tr}}

\newcommand{\mult}{\operatorname{mult}}

\newcommand{\nul}{\operatorname{null}}

\newcommand{\G}{\mathcal{G}}
\newcommand{\A}{\mathcal{A}}
\newcommand{\B}{\mathcal{B}}
\newcommand{\PP}{\mathcal{P}}

\newcommand{\QQ}{\mathcal{Q}}
\newcommand{\HH}{\mathcal{H}}

\newcommand{\oml}{{\bf m}}
\newcommand{\omlo}{{\bf m}_0}

\newcommand{\R}{\mathbb{R}} 
\newcommand{\Rnn}{\R^{n\times n}} 
\newcommand{\Rn}{\R^{n}}

\newcommand{\bx}{{\bf x}}
\newcommand{\bv}{{\bf v}}
\newcommand{\bw}{{\bf w}}
\newcommand{\bu}{{\bf u}}
\newcommand{\br}{{\bf r}}

\newcommand{\sym}{\mathcal{S}}

\newcommand{\symo}{\mathcal{S}_0}
\newcommand{\mr}{\operatorname{mr}}

\newcommand{\mro}{\operatorname{mr}_0}
\newcommand{\MRo}{\operatorname{MR}_0}
\newcommand{\M}{\operatorname{M}}
\newcommand{\maxmulto}{\operatorname{MM}_0}

\newcommand{\Mo}{\operatorname{M}_0}

\newcommand{\nc}{\operatorname{nc}}
\newcommand{\nev}{\operatorname{ne}}
\newcommand{\cyc}{\operatorname{cyc}}
\newcommand{\CC}{\mathcal{C}}

\newcommand{\dcup}{\mathbin{\,\sqcup\,}}
\newcommand{\noi}{\noindent} 
\newcommand{\x}{\times}
\newcommand{\lam}{\lambda}
\newcommand{\veps}{\varepsilon}
\newcommand{\bit}{\begin{itemize}}
\newcommand{\eit}{\end{itemize}}
\newcommand{\ben}{\begin{enumerate}}
\newcommand{\een}{\end{enumerate}}
\newcommand{\beq}{\begin{equation}}
\newcommand{\eeq}{\end{equation}}
\newcommand{\bea}{\begin{eqnarray*}}
\newcommand{\eea}{\end{eqnarray*}}
\newcommand{\bean}{\begin{eqnarray}}
\newcommand{\eean}{\end{eqnarray}}
\newcommand{\bpf}{\begin{proof}}
\newcommand{\epf}{\end{proof}\ms}
\newcommand{\bmt}{\begin{bmatrix}}
\newcommand{\emt}{\end{bmatrix}}
\newcommand{\ms}{\medskip}

\newcommand{\beqa}{\begin{array}}
\newcommand{\eeqa}{\end{array}}
\newcommand{\OL}{\overline}
\newcommand{\lc}{\left\lceil}
\newcommand{\rc}{\right\rceil}
\newcommand{\lf}{\left\lfloor}
\newcommand{\rf}{\right\rfloor}
\newcommand{\lp}{\left(}
\newcommand{\rp}{\right)}
\newcommand{\lb}{\left[}
\newcommand{\rb}{\right]}
\newcommand{\lsb}{\left\{}
\newcommand{\rsb}{\right\}}

\newcommand{\cp}{\,\Box \,}

\title{Inverse eigenvalue and related problems for hollow matrices described by graphs}
\author{F. Scott Dahlgren\thanks{Department of Mathematics and Statistics, Georgia State University, Atlanta, GA 30303, USA (fdahlgren1@student.gsu.edu).} \and  Zachary Gershkoff\thanks{Department of Mathematics, Louisiana State University, Baton Rouge, LA 70803, USA (zgersh2@lsu.edu).}\and Leslie Hogben\thanks{Department of Mathematics, Iowa State University,
Ames, IA 50011, USA and American Institute of Mathematics, San Jose, CA 95112, USA (hogben@aimath.org).}\and
Sara Motlaghian\thanks{Tri-Institutional Center for Translational Research in Neuroimaging and Data Science (TReNDS), Georgia State University, Georgia Institute of Technology, and Emory University, Atlanta, GA 30303, USA (smotlaghian1@gsu.edu)}\and Derek Young\thanks{Department of Mathematics and Statistics, Mount Holyoke College, South Hadley, MA 01075, USA (dyoung@mtholyoke.edu).}}

\begin{document}
\maketitle
\begin{abstract}
A hollow matrix described by a graph $G$ is a  real symmetric matrix having all diagonal entries equal to zero and with the off-diagonal entries governed by
the adjacencies in $G$.   For a given graph $G$, the determination  of all possible spectra of  matrices associated with $G$ is the hollow inverse eigenvalue problem for $G$. Solutions to the hollow inverse eigenvalue problems for paths and complete bipartite graphs are presented. Results for related subproblems such as possible ordered multiplicity lists, maximum multiplicity of an eigenvalue, and minimum number of distinct eigenvalues are presented for additional families of graphs. \end{abstract}

\noi {\bf Keywords}  inverse eigenvalue problem; hollow matrix; maximum multiplicity; minimum number of distinct eigenvalues; ordered multiplicity list; minimum rank; maximum nullity

\noi{\bf AMS subject classification} 05C50, 15A18, 15A29, 15B57


\section{Introduction}

 Inverse eigenvalue problems, which refer to determining all possible multisets of eigenvalues (spectra)  for matrices fitting some description, appear in various contexts throughout engineering and the mathematical, physical, biological, and social sciences. Graphs can be used to describe relationships in an application and the eigenvalues of associated matrices govern the behavior of the system.   The {\em inverse eigenvalue problem of a graph (IEP-$G$)} refers to determining the possible spectra (multisets of eigenvalues) of real symmetric matrices whose pattern of nonzero off-diagonal entries is described by the edges of a given graph.  

More precisely, a graph $G=(V(G),E(G))$ consists of a finite nonempty set of vertices $V(G)$  and a set $E(G)$ of edges, which are two element subsets of vertices. The edge $\{v_i,v_j\}$ is often denoted by $v_iv_j$.  
The {\em order} of $G$ is the number of vertices in $G$.  
The {\em set  of symmetric matrices described by $G$} is
\[\sym(G)=\{A=[a_{ij}]\in S_n(\R):  a_{ij} \ne 0 \mbox { if and only if } v_iv_j \in E(G) \mbox{ for all } 1\le i<j\le n\}, \] where $n$ denotes the order of $G$ and $S_n(\R)$ denotes that set of symmetric $n\x n$ real matrices. Thus the  IEP-$G$ is to determine the possible spectra  of the matrices in $\sym(G)$. Let $\A(G)$ denote the adjacency matrix of $G$, i.e., $a_{ij}=1$ if $v_iv_j\in E(G)$ and $a_{ij}=0$ otherwise; clearly $\A(G)\in\sym(G)$, as are several other matrices associated with a graph (such as the Laplacian).

The IEP-$G$ is a very challenging problem and was originally approached through   the study of subproblems, such at the maximum multiplicity of an eigenvalue, the minimum number of distinct eigenvalues, and  ordered multiplicity lists for eigenvalues.  
The {\em maximum multiplicity  of $G$}  is  $\M(G) = \max\{\mult_A (\lam) :  A \in \sym(G)\}$, where
$\mult_A(\lam)$ denotes the multiplicity of $\lambda$ as an eigenvalue of $A$.  Since $A\in\sym(G)$ implies $A-\lam I\in\sym(G)$, $\M(G)= \max\{\nul A :  A \in \sym(G)\}$. A related parameter is the  minimum rank of $G$, 
$ \mr(G) = \min\{\rank A :  A \in \sym(G)\}$.

Our focus is on the inverse eigenvalue problem for  hollow symmetric matrices described by a graph.  A \emph{hollow matrix} is a square matrix all of whose diagonal entries are zero.   In some applications for which spectra need to be determined, the diagonal entrees are known to be zero.  For a graph $G$ of order $n$, the matrices in \[\symo(G) = \{A \in \sym(G): a_{ii} = 0, \ 1 \leq i \leq n\}\] can also be viewed as weighted adjacency matrices of $G$.  
The \emph{hollow symmetric inverse eigenvalue problem of a graph (HIEP-$G$)}  is to determine the possible spectra of $\symo(G)$.

  The HIEP-$G$ is also a very hard problem, and related problems that shed light on the HIEP-$G$  are studied, as is the case for the IEP-$G$.
The \emph{maximum hollow nullity} and \emph{minimum hollow rank} of $G$, defined in \cite{mr0} as  \emph{maximum zero-diagonal nullity} and \emph{minimum zero-diagonal rank}, are 
\[\Mo(G) = \max \{ \nul A : A \in \symo(G)\}\mbox{ and }\mro(G) =\min \{ \rank A : A \in \symo(G)\}.\]
 As usual, $\mro(G)+\Mo(G)=|V(G)|$. 
 However, the hollow maximum nullity is not generally the maximum multiplicity of an eigenvalue of a hollow matrix, because it refers only to the maximum multiplicity of eigenvalue zero.  The \emph{maximum hollow multiplicity}  of $G$ is defined to be 
 \[\maxmulto(G) = \max \{ \mult_A(\lambda) : A \in \symo(G)\}.\]
For a symmetric matrix $A\in\Rnn$, $q(A)$ is the number of distinct eigenvalues of $A$.  
Define \[q(G)=\min\{q(A):A\in\sym(G)\}\mbox{ and }q_0(G)=\min\{q(A):A\in\symo(G)\}.\]
Let the distinct eigenvalues of a symmetric matrix $A\in\Rnn$ be denoted by
$ \mu_1(A)< \cdots < \mu_q(A)$  with multiplicities $m_1(A), \dots, m_q(A)$, respectively. 
The {\em ordered multiplicity list of $A$} is $\oml(A)=(m_1(A), \dots, m_q(A))$.  The {\em set of  ordered multiplicity lists of $G$} and the {\em set of hollow ordered multiplicity lists of $G$} are 
\[\oml(G)=\{\oml(A):A\in\sym(G) \} \mbox{ and }\omlo(G)=\{\oml(A):A\in\symo(G) \}.\]

We present full solutions to the HIEP-$G$ for paths, complete bipartite graphs, and graphs of order at most three in Section \ref{s:fam-small}. In that section we also present results on ordered mutiplicity lists, $q_0(G)$, $\Mo(G)$, and $\maxmulto(G)$ for various families and determine all possible ordered multiplicity lists for graphs of order four.  In Section \ref{s:bipart} we show that for a bipartite graph $G$, the spectrum of a hollow matrix described by $G$ is symmetric about the origin,  that $q(G)=2$ implies $q_0(G)=2$, and derive additional results for bipartite graphs. Section \ref{s:prelim} contains a variety of bounds and other tools for studying the subproblems.

The {\em trace constraint}  for hollow matrices \eqref{trace-eq} is that the sum of the eigenvalues is zero.  When the spectrum is written as distinct eigenvalues and multiplicities, i.e., $\{\mu_1^{(m_1)}, \mu_2^{(m_2)},\dots,\mu_q^{(m_q)}\}$, the trace constraint is
\beq \sum_{i=1}^q m_i\mu_i = 0. \label{trace-eq}\eeq
We will see that the trace constraint is a powerful tool for analyzing spectra of hollow matrices.
An ordered multiplicity list $(m_1,m_2,\dots,m_q)\in\omlo(G)$ is {\em hollow spectrally arbitrary} if for every set of real numbers $\mu_1<\mu_2<\dots<\mu_q$ that satisfies \eqref{trace-eq} there is some $A\in\symo(G)$ such that $\spec(A)=\{\mu_1^{(m_1)}, \mu_2^{(m_2)},\dots,\mu_q^{(m_q)}\}$.

Let $A=[a_{ij}]$ be a hollow symmetric matrix.  Given a generalized cycle $\mathcal{C}$ of $\G(A)$, let $\nc(\mathcal{C})$ denote the number of distinct cycles (of order 3 or more) in $\mathcal{C}$, and $\nev(\mathcal{C})$ denote the number of even components of $\mathcal{C}$, i.e., the number of cycles of even order at least four plus the number of edges. With a generalized cycle $\mathcal{C}$, we can associate a permutation $\pi_{\CC}$ of the vertices of $\mathcal{C}$ as follows: For each cycle in $\CC$, fix an orientation and then associate  a directed graph cycle $(v_{j_1}, v_{j_2}, \dots, v_{j_\ell})$  with the cyclic permutation $(v_{j_1} v_{j_2} \cdots\, v_{j_\ell})$.  Each edge component $\{v_{i_1}, v_{i_2}\}$  of $\CC$ is associated with the transposition $(v_{i_1}v_{i_2})$. The permutation $\pi_{\mathcal{C}}$  is  defined to be the product of these associated permutation cycles.  There are $2^{\nc(\CC)}$ different choices for the orientation of the cycles of $\CC$, and each choice yields a permutation that has the same sign as $\pi_{\CC}$, namely $(-1)^{\nev(\mathcal{C})}$.
Recall that \[p_A(x)=\det(xI-A)=x^n-S_1(A)x^{n-1}+ S_2(A)x^{n-2}+\dots\pm S_n(A)\] where $S_k(A)$ is the sum of all order $k$ principal minors of $A$ (so $S_n(A)=\det A$ and $S_1(A)=\tr A$).  Viewing $A\in\symo(G)$ as a weighted adjacency matrix of $G$, it follows from results in   \cite{Har62} 
that  $S_k(A)$ can be computed using generalized cycles of $G$:
\beq\label{eq:gencyc}
S_k(A) = \sum_{\mathcal{C} \in \cyc_k(\G(A))}
(-1)^{\nev(\mathcal{C})}2^{\nc(\mathcal{C})} a_{i_1\pi_{\mathcal{C}}(i_1)}\dots a_{i_k\pi_{\mathcal{C}}(i_k)},
\eeq
where the sum over the empty set is zero. (If $A$ has nonzero diagonal elements then the formula becomes much less useful and a looped graph must be used to describe $A$.) 
Equation \eqref{eq:gencyc} assists with the computation of the characteristic polynomial (and thus the nullity) of a specific matrix.  More generally, it is useful for computing $\mro(G)$.   
Observe that $\mro(G)\le r$ if  $G$ has no generalized cycles of order greater than $r$, and $\mro(G)= r$ if  $G$ has no generalized cycles of order greater than $r$ and $G$ has a unique generalized cycle of order $r$.

When studying the ranks of  matrices in $\sym(G)$, one studies only minimum rank, because it is well known and easy to see that the maximum rank is the order of the graph $G$, and every rank between the minimum and maximum ranks is realizable.  However, there are many graphs for which the maximum rank of a hollow matrix described by a graph $G$ must be  less than the order of $G$.
 The {\em  maximum hollow rank} of a graph $G$  is
$\MRo(G)=\max\{ \rank A :~A\in\symo(G)\}.$

\begin{thm}\label{maxrank}{\rm\cite{mr0}} For a graph $G$, $\MRo(G)$  is the maximum order of a generalized cycle of $G$.
\end{thm}

Maximum nullity, minimum number of distinct eigenvalues, and ordered multiplicity lists all provide information that can in some cases be used to solve the inverse eigenvalue problem for a specific graph or family of graphs. Recently   the Strong Spectral Property (SSP) and the Strong Multiplicity Property (SMP) were introduced in \cite{genSAP}. These  tools, which extend the Strong Arnold Property for nullity, have been the focus of much recent research in the IEP-$G$.  While we do not provide the details here, a fundamental principle of a strong property is subgraph monotonicity.   In the case of the Strong Arnold Property, this means that the existence of a matrix  $A\in \sym(G)$ that has the Strong Arnold Property implies the existence of $B\in\sym(H)$ with $\nul B\ge \nul A$ for any graph $H$ having $G$ as a subgraph.  For the Strong Spectral Property,  if $A\in \symo(G)$ has this property and $G$ is a subgraph of $H$, then there is a matrix $B\in\symo(H)$ such that $\spec(A)\subseteq\spec(B)$.
This naturally raises the question of the existence of strong properties for hollow matrices. Unfortunately,  hollow matrices do not seem well suited to strong properties, as we show in the next example. 

\begin{ex}\label{ex:no-hstrong}
For any graph  $G$ of order $n$, $G\circ K_1$ is constructed from $G$ by adding a leaf to each vertex of $G$. Note that $\Mo(G\circ K_1)=0$, i.e., $\mult_A(0)=0$ for $A\in\symo(G\circ K_1)$,  because $G\circ K_1$ has a unique generalized cycle of order $2n$.  For example, $\mult_A(0)=n-2$ for $A\in K_{1,n-1}$ and $\mult_B(0)=0$ for $B\in\symo(K_{1,n-1}\circ K_1)$. This implies that there cannot be a nonzero numerical ``strong'' property below $\Mo$  that is induced subgraph monotone, and there cannot be a strong spectral property for hollow matrices.
\end{ex}

 The general failure of strong properties for hollow matrices is not surprising, because the proofs of subgraph monotonicity for strong properties are based on small perturbations, using the fact that for a sufficiently small perturbation a nonzero entry remains nonzero but a zero entry can be changed to nonzero.  However, some perturbation techniques can still be used, as in the proof of  Proposition \ref{p:allsimp}, where we show the ordered multiplicity list with all entries equal to one can alwaysde be realized.


\section{Preliminary results}\label{s:prelim}

In this section we present examples and  general results about the HIEP-G and related parameters.

\begin{prop}\label{q02specarb}
Let $G$ be a graph of order $n\ge 2$.  If $(m_1,m_2)$ is an ordered multiplicity list of $G$, then $(m_1,m_2)$ is hollow spectrally arbitrary.  
\end{prop}
\bpf When there are only two distinct eigenvalues,  the trace constraint   \eqref{trace-eq}  implies the spectrum is determined by one eigenvalue and its multiplicity.  That is, for $\spec(A)=\{\mu_1^{(k)},\mu_2^{(n-k)}\}$, $\mu_2=-\frac k{n-k}\mu_1$ (observe that both $\mu_1$ and $\mu_2$ are nonzero).  Given that one specific spectrum $\spec(A')=\{{\mu'}_1^{(k)},{\mu'}_2^{(n-k)}\}$ is realized by $A'\in\symo(G)$, the matrix $A=\frac{\mu_1}{{\mu'}_1} A'$ has $\spec(A)=\{\mu_1^{(k)},\mu_2^{(n-k)}\}$.
\epf

\begin{ex}\label{q0-Kn}
Note that $\spec(\A(K_n))=\{n-1,(-1)^{(n-1)}\}$. By considering this matrix and its negative, we see that $q_0(K_n)=2$  and $(n-1,1), (1,n-1)\in\omlo(K_n)$ for $n\ge 2$.  By Proposition \ref{q02specarb}, $(n-1,a)$ and $(1,n-1)$ are hollow spectrally arbitrary.
\end{ex}

\begin{rem}\label{o:dunion}
If $G=G_1\dcup G_2$, then $\spec(G)=\spec(G_1)\cup\spec(G_2)$ (where the union of spectra is a multiset union).  Thus, it is common to focus on connected graphs when studying the full HIEP-$G$.  However, it is not necessarily easy to determine $q_0(G_1\dcup G_2)$ from $q_0(G_1)$ and $q_0(G_2)$ (see Section \ref{ss:q0}). 
\end{rem}

The \emph{maximum semidefinite nullity} of a graph $G$ is \[\M_+(G)=\max\{\nul A : A\in\sym(G)\mbox{ and $A$ is positive semidefinite}\}.\] Let $A$ be a positive semidefinite matrix of rank $d$.  Then there is an $d\x n $ matrix $R=[\br_1,\dots,\br_n]$ such that $A=R^TR$. If in addition $A\in\sym(G)$, then the column vectors $\br_i$ are called an \emph{orthogonal representation} of $G$ of dimension $d$. 
It is well-known that the least  $d$ such that $G$ has an orthgonal representation of dimension $d$ is $d=n-\M_+(G)$ for a graph $G$ of order $n$.
\begin{lem}\label{r:m1} Let $G$ be a graph with no isolated vertices. 
Then  there exists a  matrix $B\in\symo(G)$ such that $m_1(B)=\M_+(G)$.
\end{lem}
\bpf Let $m=\M_+(G)$. Choose a  positive semidefinite matrix $A\in\sym(G)$ such that $\nul A=m$. Then there is an $(n-m)\x n $ matrix $R=[\br_1,\dots,\br_n]$ such that $A=R^TR$.  Define $A'=D^T AD$ where $D=\diag\lp\frac{1}{\|\br_1\|},\dots,\frac{1}{\|\br_n\|}\rp$.  Observe that every diagonal entry of $A'$ is 1, and $\nul A'=m$.  Define $B=A'-I$.  Note that $B\in\symo(G)$, $\mu_1(B)=-1$, and $m_1(B)=m$. 
\epf   

Note that the process of constructing $B$ in Lemma \ref{r:m1}  involves conjugation, which preserves the nullity but not the spectrum.  Typically when the conjugation process is applied to a matrix, the  other eigenvalues are perturbed, and multiple eigenvalues may split into several  simple eigenvalues.  However, Lemma \ref{r:m1} can be leveraged to determine other multiplicities for certain ordered multiplicity lists of bipartite graphs (see Corollary \ref{bip-k1k}). 


\subsection{$q_0(G)$}\label{ss:q0}
In this section we focus on the minimum number of distinct eigenvalues of hollow matrices described by a graph.  

\begin{obs} For any graph $G$, $q(G)\le q_0(G)$. It is known 
 that $q(G)=1$ if and only if $G$ has no edges {\rm \cite{AACFMN}},  
  so $q_0(G)=1$ if and only if $G$ has no edges.  For any graph $G$ of order $n$, $q_o(G)\ge \frac n{\maxmulto(G)}$.
\end{obs}

The ease of combining spectra does not apply when determining  the minimum number of distinct eigenvalues of a disjoint union, because the spectra may or may not align. Proposition \ref{K3duKr} shows that $q_0(G\dcup H) > \max\{q_0(G), q_0(H)\}$ is possible.  Examples where $q_0(G\dcup H) = \max\{q_0(G), q_0(H)\}$ are easily constructed when ordered multiplicity lists are  arbitrary within shared constraints, such as spectral symmetry about the origin. 

\begin{prop}\label{K3duKr} If $r$ is not a multiple of $3$, then $q_0(K_3\dcup K_r)=3$. 
\end{prop}
\bpf The matrix $A\in\symo(K_3\dcup K_r)$ is of the form $A=A_3\oplus A_r$, where $A_3\in\symo(K_3)$ and $A_r\in\symo(K_r)$. Thus, $\spec(A)=\spec(A_3)\cup\spec(A_r)$. If $q(A_3)\ge 3$  or $q(A_r)\ge 3$, then $q(A)\ge 3$.  So suppose $q(A_3)=2$ and $q(A_r)=2$. Observe that  $\spec(A_3)=\{\lambda^{(2)},-2\lambda\}$ for some $\lambda$ and $\spec(A_r)=\{\mu^{(k)},-\frac k {r-k}\mu^{(r-k)}\}$ for some $\mu$.  To have $q(A)=2$, necessarily (i) $\lambda=\mu$ and $2\lambda=\frac k {r-k}\mu$, or (ii) $-2\lambda=\mu$ and $\lambda=-\frac k {r-k}\mu$.  Then  $2r=3k$ in case (i) and $r=3k$ in case (ii).  In either case, $q(A)=2$ implies $r$ is a multiple of $3$.
\epf

The proof of the next result is the same as the proof of \cite[Proposition 2.5]{AACFMN}.
\begin{prop} For any graph, $q_0(G)\le \mro(G)+1$.
\end{prop}
\bpf Choose $A\in\symo(G)$ with $\rank A=\mro(G)$. Then $A$ has at most $\mro(G)$ nonzero eigenvalues and $\mro(G)+1$ distinct eigenvalues.
\epf

\begin{prop}\label{zero-eval} Let $G$ be a graph of order $n$ that is not the empty graph such that $\MRo(G)<n$.  Then $q_0(G)\ge 3$. 
\end{prop}
\bpf Let $A\in\symo(G)$.  Since $G\ne \OL{K_n}$, $A$ has  at least 1 positive and at least 1 negative eigenvalue.  Since $\MRo(G)<n$, $0\in\spec(A)$.  Thus $q(A)\ge 3$.
\epf

\begin{prop}\label{q02delv}
Let $G$ be a graph of order $n\ge 4$ such that $G$ has the ordered multiplicity list $(r,n-r)$ with $2\le r\le n-2$. 
For every vertex $v \in V(G - v)$, there exists $B \in \symo(G - v)$ such that $
\spec(B) = \{ (-\frac{n-r}{r})^{(r-1)}, -\frac{n-2r}{r},1^{(n-r-1)}
\}$. 
\end{prop}

\begin{proof} There is a matrix $ A \in \symo(G) $ such that $\spec(A) = \{
(-\frac{n-r}r)^{(r)}, 1^{(n-r)} \} $ by  Proposition \ref{q02specarb}. 
Let $B = A(v)$. Then $B \in \symo(H)$. Define the eigenvalues
of $B$ to be $\mu_1 \le \mu_2\le \dots\le \mu_{n-1}$. By interlacing,   
\[
-\frac{n-r}{r} \leq \mu_1 \leq \dots \leq \mu_{r-1}  \leq -\frac{n-r}{r} \leq
\mu_r \leq 1 \leq \mu_{r+1} \leq \dots \leq \mu_{n-1} \leq 1.
\]
Hence $\mu_1 = \mu_2 = \cdots = \mu_{r-1} = -\frac{n-r}{r}$ and $ \mu_{r+1} =
\cdots = \mu_{n-1} = 1$. Thus $(n-r-1)-\frac{n-r}{r} ( r-1 ) +\mu_r = 0$ which implies 
$\mu_r = -(n-r-1) + (r-1)\lp\frac{n-r}{r}\rp =   \frac{2r-n}{r}$.
\end{proof}

\begin{cor}
Let $G$ be a graph of order $n\ge 3$ such that $q_0(G) = 2$. Then $q_0(G - v ) \leq 3
$.  If $\maxmulto(G) = n - 1$, then $q_0(G - v ) = 2$.
\end{cor}
\bpf Let  $A\in\symo(G)$ have $\oml(A)=(r,n-r)$.   If $r=1$, then $\oml(A(v))=(1,n-2)$ by interlacing (and neither eigenvalue is $0$ by the trace inequality),  so $q_0(G-v)=2$; the case $r=n-1$ is analogous. If $n=3$, then $q_0(G)=2$ implies $G=K_3$ and $q_0(K_{2})=2$.  So assume  that $n\ge 4$ and $2\le r\le n-r$.  Then $q(A(v))=3$ by Proposition \ref{q02delv} (since $-\frac{n-2r}{r}\ne 1, -\frac{n-r}r$).  Thus $q_0(G)\le 3$.  \epf

Next we establish results that allow extending low values of $q_0(G)$ to low values of $q_0(G\cp K_2)$ when $q_0(G)$ is realized by a matrix whose spectrum is symmetric about the origin, as is the case for bipartite graphs (see Theorem \ref{t:bip-spec-sym}).  

 \begin{lem}\label{l:cartprod} Suppose $G$ is a graph of order $n$ such that there exists $A\in\symo(G)$ with $\spec(A)=\{-\mu_k^{(m_k)},\dots,$ $-\mu_1^{(m_1)},0^{(m_0)}, \mu_1^{(m_1)},\dots,\mu_k^{(m_k)}\}$ where $0<\mu_1<\dots <\mu_k$ and $m_0\ge 0$ (with $m_0=0$ signifying that $0$ is not an eigenvalue of $A$).  Define $B=\mtx{A & I_n\\I_n & -A}$.  Then  $B\in\symo(G\cp K_2)$ and  \[\spec(B)=\lsb-\sqrt{\mu_k^2+1}^{(2m_k)},\dots,-\sqrt{\mu_1^2+1}^{(2m_1)},-1^{(m_0)},1^{(m_0)}, \sqrt{\mu_1^2+1}^{(2m_1)},\dots,\sqrt{\mu_k^2+1}^{(2m_k)}\rsb.\]
\end{lem}
\bpf It is immediate that $B\in\symo(G\cp K_2)$.
Furthermore, $\spec(A^2+I)=\{1^{(2m_0)}, (\mu_1^2+1)^{(2m_1)},\dots, (\mu_k^2+1)^{(2m_k)}\}$.  
 The characteristic polynomial of $B$ is
\[ \det(xI_{2n}-B)=\det\mtx{xI_n-A &-I_n\\-I_n & xI_n+A}=\det((xI_n-A)(xI_n+A)-I_n)=\det(x^2 I_n-(A^2+I_n)).\]
Thus $\spec(B)=-\spec(B)$ and is as stated. 
\epf

The next result is immediate from Lemma \ref{l:cartprod}.
 \begin{cor}\label{p:cartprod} Let $G$ be a graph of order $n$ such that there exists $A\in\symo(G)$ with $q(A)=q_0(G)$ and $\spec(A)=-\spec(A)$.  Then  there exists $B\in\symo(G\cp K_2)$ having $\spec(B)=-\spec(B)$,  $0\not\in \spec(B)$, and $q(B)= q_0(G)+a$ where $a=0$ if $0\not\in\spec(A)$ and $a=1$ if  $0\in\spec(A)$.
\end{cor}

The process of taking the Cartesian product with $K_2$ can be repeated without raising $q_0$ beyond the first Cartesian product. 
\begin{cor}\label{c:qo2-cartprod}
If there exists $A\in\symo(G)$ with $q(A)=q_0(G)$ and $\spec(A)=-\spec(A)$, then   $q_0(G\cp K_2\cp\cdots$ $\cp K_2)\le q_0(G)+1$.  If there exists such a matrix $A$ with $0\not\in\spec(A)$, then $q_0(G\cp K_2\cp\cdots\cp K_2)\le q_0(G)$.  Thus  $q_0(Q_d)=2$ where $Q_d$ is the $d$-dimensional hypercube.
\end{cor}

\subsection{ $\Mo(G)$ and $\maxmulto(G)$}\label{ss:MM0}
In this section, we focus on the maximum hollow nullity $\Mo(G)$ and maximum hollow multiplicity $\maxmulto(G)$ of a graph $G$. 
It is well known that adding a dominating vertex to a graph with no isolated vertices does not change the minimum rank (or positive semidefinite minimum rank):  if $A\in\sym(G)$ and $\rank A=\mr(G)$, choose $\bv$ such that every entry of $A\bv$ is nonzero and construct the matrix $B$  as in the next proof.  It is not so simple to find such a $\bv$ when the diagonal must be zero (and  it is not always possible). 

\begin{lem}\label{l:Zm-domvtx} Let $A$ be an $n\x n$ hollow symmetric matrix that has a nonzero eigenvalue $\lam$ with $\mult_A(\lam)\ge 2$ and an eigenvector $\bw$ for $\lam$ that has every entry nonzero.  Then there exists a vector $\bv\in\Rn$ such that both $\bv^TA\bv=0$ and every entry of  $A\bv$ is nonzero. Furthermore, $B=\mtx{A & A\bv\\ \bv^TA & \bv^TA\bv}\in \symo(\G(A)\vee K_1)$ and $\rank B = \rank A$. \end{lem}
\bpf
 Denote the eigenvalues of  $A$ by  $\lam_1=\lam,\lam_2=\lam,\lam_3,\dots,\lam_{n}$ (no ordering implied).  Let $\bu$ be a multiple of $\bw$ of length one, and choose a basis of orthonormal eigenvectors $\bx_1=\bu,\bx_2,\dots, \bx_{n}$ with $A\bx_i=\lam_i\bx_i$.   For a real number $a$ with $-\sqrt 2<a<\sqrt 2$, define \[\bv_a= a \bx_1+\sqrt{2-a^2}\bx_2+\bx_3+\dots+\bx_{n}.\]  
 Since  $\bx_1,\dots, \bx_{n}$ are orthonormal eigenvectors, \[A\bv_a=a\lam\bu+\sqrt{2-a^2}\lam\bx_2+\lam_3\bx_3+\dots+\lam_{n}\bx_{n}\] and \[\bv_a^TA\bv_a=a^2\lam+(2-a^2)\lam+\lam_3+\dots+\lam_{n}=\tr(A)=0.\] To have all entries of $A\bv_a$ nonzero, choose $a$ to avoid the solutions to the equations 
 \[a\lam(\bu)_j+\sqrt{2-a^2}\lam(\bx_2)_j+\lam_3(\bx_3)_j+\dots+\lam_{n}(\bx_{n})_j\]
 for $j=1,\dots,n$ (note that each equation is nontrivial because $\lam(\bu)_j\ne 0$). The last statement is immediate, given $\bv=\bv_a$ with appropriate choice of $a$.
  \epf

\begin{lem}\label{l:MMo-domvtx} Let $G$ be a graph with no isolated vertices such that $(m_1,\dots,m_r)\in \omlo(G)$ and let $H=G\,\vee K_1$.  Then there exists $B\in\symo(H)$ with $m_1(B)=m_1+1$, and $\maxmulto(H)\ge m_1+1$. \end{lem}
\bpf Let $A\in\symo(G)$ such that $\oml(A)=(m_1,\dots,m_r)$,  and let $\mu=-\mu_1(A)$.  Define  $A'=A+\mu I_n$ so $A'$ is positive semidefinite. Then there is an $(n-m_1)\x n $ matrix $R=[\br_1,\dots,\br_n]$ such that $A'=R^TR$.  Then $\br_i^T\br_i=\mu$ for $i=1,\dots,n$ since every diagonal entry of $A'$ is equal to $\mu$,  and for $j\ne i$, $\br_i^T\br_j\ne 0$ if and only if $ij\in E(G)$.  Choose an $(n-m_1)$-vector $\br_{n+1}$ such that $\br_i^T\br_{n+1}\ne 0$ for $i=1,\dots,n$ and $\br_{n+1}^T\br_{n+1}=\mu$ (a random vector normalized to have length $\mu$ will work with high probability).  Define $R'= [\br_1,\dots,\br_n,\br_{n+1}]$ and $A''=R'^TR'$. Then $A''\in\sym(H)$ and every diagonal entry of $A''$ is $\mu$, so $B=A''-\mu I_{n+1}\in\symo(H)$.  Since $\rank A''=\rank A'=m_1$, $\nul A''=m_1+1$ and $m_1(B)=m_1+1$.  Thus $\maxmulto(H)\ge m_1+1$.
\epf

\begin {cor}\label{c:domvtx} If there exists $A\in\symo(G)$ such that $\maxmulto(A)=\mult_A(\mu_1(A))$, then $\maxmulto(G\vee K_1)=\maxmulto(G)+1$.  
\end{cor}
\bpf Suppose there exists $A\in\symo(G)$ such that $\maxmulto(A)=\mult_A(\mu_1(A))$.  Then $\maxmulto(G\vee K_1)\ge \maxmulto(G)+1$ by Lemma \ref{l:MMo-domvtx}.  Interlacing implies $\maxmulto(G\vee K_1)\le \maxmulto(G)+1$. \epf

Note that the process of constructing $B$ in Lemma  \ref{l:MMo-domvtx}  preserves the nullity but not the spectrum.  Interlacing provides some control of other multiple eigenvalues but  multiplicities may each be reduced by one and the number of distinct eigenvalues may increase. 


\subsection{Simple eigenvalues and ordered multiplicity lists}\label{ss:oml}

It is well known that every graph allows a matrix $A\in\sym(G)$ having every eigenvalue simple.  This is not the case for a graph $G$ of order $n$ that has $\MRo(G)\le n-2$, since $\mult_A(0)\ge n-\MRo(G)\ge 2$ for $A\in\symo(G)$.  In this section we show that every graph $G$ allows a matrix of maximum hollow rank in which every nonzero eigenvalue is simple.

Recall that $\M(P_n)=1$, 
from which it is immediate that  $\maxmulto(P_n)=1$ and $q_0(P_n)=n$.  
Thus $(1,1,\dots,1)\in \symo(P_n)$ for all $n$. 

\begin{lem}\label{l:simpcycl} For any cycle $C_n$, there is a matrix $A\in \symo(C_n)$ such that $\oml(A)=(1,1,\dots,1)$ and every eigenvalue of $A$ is nonzero.
\end{lem}
\bpf Let $P_n$ be obtained from $C_n$ by deleting the edge $\{1,n\}$.  Choose $A\in\symo(P_n)$. Note that the eigenvalues of $A$ are distinct.  
Let $H$ be the graph with $V(H)=V(G)$ and $E(H)=\{1,n\}$.  Define $A_{\veps}=A+\veps \A(H)$. 
Then for $\veps>0$,  $A_{\veps}\in\symo(C_n)$, and the eigenvalues of $A_{\veps}$ are distinct  for $\veps$ sufficiently small by continuity.  If $n$ is even, then the eigenvalues of $A$ are all nonzero, so the eigenvalues of $A_{\veps}$ are nonzero for $\veps$ sufficiently small. If $n$ is odd, then $\mro(C_n)=n$, so the eigenvalues of $A_{\veps}$ are all nonzero. \vspace{-5pt}
\epf

\begin{prop} \label{p:allsimp} Let $G$ be a graph of order $n$.  There is a matrix $A\in\symo(G)$ that has $\MRo(G)$ simple nonzero eigenvalues.  If $\MRo(G)\ge n-1$, then there is a matrix $A\in \symo(G)$ such that $\oml(A)=(1,1,\dots,1)$.
\end{prop}
\bpf Choose a generalized cycle $\mathcal{C}$ of order $\MRo(G)$.  Note that $\mathcal{C}$ is the disjoint union of $k\ge 1$ cycles $C_{n_i}$ (where an edge is denoted by $C_2$). For each component $C_{n_i}$, there is a matrix $A_i\in \symo(C_{n_i})$ that has $n_i$ distinct nonzero eigenvalues.  Choose nonzero $\alpha_i\in\R$ such that the set of eigenvalues of the matrices $\alpha_iA_i$ has $n$ distinct nonzero elements.  Then the eigenvalues of $A=\alpha_1A_1\oplus\dots\oplus \alpha_kA_k$ are distinct and nonzero.  Let $H$ be the graph with $V(H)=V(G)$ and $E(H)=\{e\in E(G):e\not\in E(\mathcal{C})\}$.  Define $A_{\veps}=A+\veps \A(H)$. Then for $\veps>0$,  $A_{\veps}\in\symo(G)$, and the eigenvalues of $A_{\veps}$ are distinct and nonzero for $\veps$ sufficiently small by continuity.  The last statement is immediate.
\epf

The minimum number of distinct eigenvalues of a graph is an active area of research, whereas the maximum number of  distinct eigenvalues of a matrix in $\sym(G)$ is the order of the graph $G$.  For hollow matrices, the maximum number of distinct eigenvalues is provided by Proposition \ref{p:allsimp}.

\begin {cor}\label{c:maxdistevals} Let $G$ be a graph of order $n$.  The maximum number of distinct eigenvalues of  a matrix in $\symo(G)$ is $\min\{\MRo(G)+1,n\}$.
\end{cor}


\section{Bipartite graphs}\label{s:bipart}

For a hollow matrix described by a bipartite graph, the spectrum is symmetric about the origin,  i.e., $\spec(A)=-\spec(A)$ for any  $A\in\symo(G)$. That result (Theorem \ref{t:bip-spec-sym}) is established in this section, together with numerous consequences and other results specific to bipartite graphs.  Results in this section are used to solve the hollow IEP-$G$ for all complete bipartite graphs $K_{m,n}$ in Section \ref{s:fam-small}. 
An {\em order $k$ principal minor of $A$} is the determinant of a $k\times k$ principal submatrix of $A$.

\begin{rem} 
\label{obs:bp_even_rank}\label{obs:bp_odd_zero}
Suppose $G$ is a bipartite graph, and let $A \in \symo(G)$. Then $A$ has even rank, because \[A=\left[\begin{array}{cc}
	0&B\\
	B^T&0\\
	\end{array}\right]\]
for some matrix $B$, so the rank of $A$ is twice the rank of $B$.
  Let $S \subset V(G)$ such that $|S| = k$, with $k$ odd.  Then, the induced subgraph $G[S]$ is again bipartite.  Since $\rank A[S]$ must be even, $A[S]\in\symo(G[S])$ cannot be of full rank.  Therefore, the corresponding principal minor is zero. Hence, every odd order principal minor of $A$ is zero.
\end{rem}

\begin{thm}\label{t:bip-spec-sym}
Suppose $G$ is a bipartite graph on $n$ vertices, and $A \in \symo(G)$.  Then $\spec(A) = -\spec(A)$. 
\end{thm}
\begin{proof}
Let $p_A(x)=x^n+a_{n-1}x^{n-1}+\dots+a_1x+a_0$ be the characteristic polynomial of $A$. Recall that $a_{n-k}=(-1)^k S_k(A)$ where $S_k(A)$ is the sum of all order $k$ principal minors of $A$.
From Remark \ref{obs:bp_odd_zero}, the characteristic polynomial of $A$ is
$p_A(x)=x^n+S_2(A) x^{n-2} + S_4(A) x^{n-4}+ \dots  $
because $S_k(A) = 0$ for odd $k$. 
First, we suppose $n$ is even. 
We reparameterize $p_A$ with $y=x^2$:
\[p_A(x)=f_A(y):=y^{n/2}+S_2(A) y^{n/2-1}+S_4(A) y^{n/2-2} + \dots + S_{n/2}(A).\]
Next, we factor 
\[f_A(y)=(y-\alpha_1)(y-\alpha_2)\cdots(y-\alpha_{n/2}),\]
where $\alpha_1, \alpha_2, \dots, \alpha_{n/2}$ are non-negative real numbers because $A$ is a real symmetric matrix.  Since  $p_A(x^2)=f_A(y)$,
\bea p_A(x)&=&(x^2-\alpha_1)(x^2-\alpha_2)\cdots(x^2-\alpha_{n/2})\\
&=&(x-\sqrt{\alpha_1})(x+\sqrt{\alpha_1})(x-\sqrt{\alpha_2})(x+\sqrt{\alpha_2})\cdots(x-\sqrt{\alpha_{n/2}})(x+\sqrt{\alpha_{n/2}}).\eea
 Therefore, $\spec(A) = -\spec(A)$ for even $n$.

Next, we suppose $n$ is odd. Then, from Remark \ref{obs:bp_odd_zero}, the characteristic polynomial of $A$ is
\bea p_A(x)&=&x^n+S_2(A) x^{n-2} + S_4(A) x^{n-4}+ \dots + S_{n-1}(A)x\\
&=&x(x^{n-1}+S_2(A) x^{n-3} + \dots + S_{n-1}(A)).\eea
We can therefore define $f_A(y)$ where $p_A(x)=x f_A(y)$ with $y=x^2$ and $f_A(y)$ factors as in the case where $n$ is even. 
As the remaining eigenvalue is zero,  $\spec(A) = -\spec(A)$ for odd $n$.
\end{proof}

\begin{cor}\label{bip-q3} If $G$ is bipartite, $G\ne \OL{K_n}$, and the order of $G$ is odd, then $q_0(G)\ge 3$.
\end{cor}

\begin{cor}\label{bip-k1k}
Suppose $G$ is a bipartite graph  of order $2k+1$ such that $G$ has no isolated vertices and  $\M_+(G)=k$.  Then $(k,1,k)\in\omlo(G)$.
\end{cor}
\bpf By Lemma \ref{r:m1}, there exists $B\in\symo(G)$ such that $m_1(B)=k$.  Since $\spec(B)=-\spec(B)$, $\spec(B)=\{(-\lam)^{(k)},0,\lam^{(k)}\}$.\vspace{-8pt}
\epf
 Note that  Lemma \ref{r:m1}   can also be used to show there is an ordered multiplicity list of the form\break $(m_1,m_2,\dots,m_{2},m_1)\in\omlo(G)$ where $m_1=\M_+(G)$  whenever  $G$ is bipartite with no isolated vertices.

\begin{prop}\label{p:no121}
 Let $G$ be a graph of even order $n$ such that $\mro(G)\ge n-1$ and there exists $v\in V(G)$ such that $G-v$ is bipartite.
 If $m_1+\dots+m_{r-1}=\frac{n-2}2=m_{r+1}+\dots+m_s$, then $(m_1,\dots,m_{r-1},2,m_{r+1},\dots ,m_s)\not\in\omlo(G)$.
\end{prop}
\bpf Let $k=\frac{n-2}2$. Suppose $m_1+\dots+m_{r-1}=\frac{n-2}2=m_{r+1}+\dots+m_s$ and $A\in\symo(G)$ has $(m_1,\dots,m_{r-1},2,$ $m_{r+1},\dots ,m_s)\in\oml(A)$.  Denote the eigenvalues of $A$ by $\alpha_1\le\dots\le \alpha_{k-1}<\alpha_k=\alpha_{k+1}<\alpha_{k+2}\le\dots\le\alpha_n$. Let $B=A(v)$, and denote its eigenvalues in order by $\beta_1\le\dots\le \beta_{n-1}$. Since $G-v$ is bipartite, $\beta_i=-\beta_{n-i}$ for $i=1,\dots,n-1$.  Since $n-1$ is odd, $\beta_k=0$.  By interlacing, $\alpha_k\le \beta_k\le \alpha_{k+1}$.  Since $\alpha_k=\alpha_{k+1}$, $\alpha_k=\alpha_{k+1}=\beta_k=0$.  Thus $\rank A\le n-2$, contradicting $\mro(G)\ge n-1$.
\epf

If $G$ is a bipartite graph and $A\in\symo(G)$, then $\spec(A)$ has the form \[\{\mu_1^{(m_1)},\dots,\mu_k^{(m_k)}, (-\mu_k)^{(m_k)},\dots,(-\mu_1)^{(m_1)}\}\mbox{ or } \{\mu_1^{(m_1)},\dots,\mu_k^{(m_k)},0^{(m_{k+1})},(-\mu_k)^{(m_k)},\dots,(-\mu_1)^{(m_1)}\}.\]  If for every $\mu_1<\dots<\mu_k$ there is an $A\in\symo(G)$ with this spectrum, then the ordered multiplicity list \[(m_1,\dots,m_k,m_k,\dots,m_1)\mbox{ or }(m_1,\dots,m_k,m_{k+1},m_k,\dots,m_1)\] is called {\em bipartite spectrally arbitrary} for $G$.

\begin{rem}
Suppose $G$ is a bipartite graph and $A\in\symo(G)$ has $\oml(A)=(m_1,m_1)$ or $\oml(A)=(m_1,m_2,m_1)$. Then $\oml(A)=(m_1,m_1)$ or $(m_1,m_2,m_1)$ is bipartite spectrally arbitrary for $G$ because $\spec(A)= \{(-\lambda)^{(m_1)},$ $\lambda^{(m_1)}\}$ or  $\spec(A)=\{(-\lambda)^{(m_1)},$ $0^{(m_2)},\lambda^{(m_1)}\}$ and any choice of the nonzero eigenvalue can be realized by multiplying $A$ by a scalar.
\end{rem}

The next result follows from Theorem \ref{t:bip-spec-sym} and 
Corollary \ref{c:qo2-cartprod}.

\begin{cor} Let $G$ be a bipartite graph. Then $q_0(G\cp K_2\cp\cdots\cp K_2)\le q_0(G)+1$.  If there exits $B\in\symo(G)$ with $q(B)=q_0(G)$ and $0\not\in\spec(B)$,  then $q_0(G\cp K_2\cp\cdots\cp K_2)\le q_0(G)$.   
If  $q_0(G)=2$, then $q_0(G\cp K_2\cp\cdots\cp K_2)=2$.  
\end{cor}

A real (respectively, complex) $n\x n$ matrix $U$ is orthogonal (respectively, unitary) if $U^TU=I_n$ (respectively, $U^*U=I_n$). 
Orthogonal matrices play an important role in the study of bipartite graphs $G$ having $q_0(G)=2$, as they do  in the study of $q(G)$ for arbitrary graphs \cite{AACFMN}.  A graph $G$ is {\em potentially hollow  orthogonal}  if there is an orthogonal  matrix $U\in \symo(G)$.  


\begin{rem}\label{r:bipart-q0-2}
Let $G$ be a graph.  Since the only possible eigenvalues of an orthogonal matrix are $1$ and $-1$, it is immediate that if $U\in\symo(G)$ is orthogonal, then $q_0(G)=2$ (since  $G$ must have an edge to allow an orthogonal matrix).  Now suppose that $G$ is  a bipartite graph, the  order of $G$ is $n=2k$, and $q_0(G)=2$. For $A\in\symo(G)$ with $q(A)=2$, $\spec(A)=\{(-\mu)^{(k)},\mu^{(k)}\}$  and  $\frac 1\mu A$ is orthogonal (because any real symmetric matrix with all eigenvalues equal to $-1$ or $1$ is orthogonal).  Therefore, $G$ is potentially hollow  orthogonal.   
Thus a bipartite graph $G$ is potentially hollow  orthogonal if and only if $q_0(G)=2$. 
\end{rem}

When studying the IEP-$G$ (without the restriction that matrices are hollow), we have the opportunity to shift (add $cI_n$) as well as  to scale (multiply by a scalar), so any matrix $A\in\sym(G)$ with $q(A)=2$ can be modified to obtain an orthogonal matrix in $\sym(G)$.  Thus $q(G)=2$ implies $G$ is potentially orthogonal.  However, this need not be true for hollow matrices of  graphs that are not bipartite.  
Recall that $q_0(K_n)=2$, and choosing $n$ odd implies $K_n$ is not potentially orthogonal.

Although the emphasis in this section is on bipartite graphs, orthogonal matrices play a broader role, as illustrated in the next result.

\begin{prop}\label{oml-rr} Let $G$ be a graph of order $2r$. Then $(r,r)\in\omlo(G)$ if and only $G$ is hollow potentially orthogonal.
\end{prop}
\bpf Suppose there exists $A\in\symo(G)$ with $\oml(A)=(r,r)$.  Then the trace condition \eqref{trace-eq}
 implies $\spec(A)=\{(-\lam)^{(r)},\lam^{(r)}\}$.  Then $\frac 1 \lam A$ is orthogonal.

If $A\in\symo(G)$ is orthogonal, then every eigenvalue is equal to  1 or $-1$, and the trace condition implies $\oml(A)=(r,r)$.
\epf

\begin{cor}\label{leaf-no-oml-rr}  Let $G$ be a graph such that vertices $v$ and $w$  have a unique common neighbor.    Then $(r,r)\not\in\omlo(G)$.  In particular, if $G$ is a connected graph of order at least $3$ that has a leaf, then $(r,r)\not\in\omlo(G)$.
\end{cor}
\bpf Let $u$ be the unique common neighbor of  $v$ and $w$.  
Let $A\in\symo(G)$.  Then $(A^2)_{vw}= a_{vu}a_{uw}\ne 0$, so $A$ is not orthogonal.  So $(r,r)\not\in\omlo(G)$ by Proposition \ref{oml-rr}.  If $G$ is  connected with $|V(G)|\ge 3$ and has a leaf $v$, then the  neighbor of $v$ is the unique common neighbor of at least two vertices of $G$.
\epf

The {\em pattern of a matrix} $B=[b_{ij}]$ is defined by $\PP(B)=[\beta_{ij}]$ where $\beta_{ij}=*$ if $b_{ij}\neq 0$ and $\beta_{ij}=0$ if $b_{ij}=0$ (note $B$ need not be square). The {\em qualitative class of an $r\x r'$} nonzero pattern matrix $\B=[\beta_{ij}]$ is $\QQ(\B)=\{B\in \R^{r\x r'}: \PP(B)=\B\}$. Let $G$ be a bipartite graph with partite sets $X=\{x_1,\dots,x_{r}\}$ and $X'=\{x'_1,\dots,x'_{r'}\}$. The {\em bigraph pattern} of $G$ is the $r\x r'$ nonzero pattern matrix $\B_G=[\beta_{ij}]$ having $\beta_{ij}=*$ if $x_ix'_j\in E(G)$ and $\beta_{ij}=0$ if $x_ix'_j\not\in E(G)$.  A pattern matrix $\B$ is {\em potentially orthogonal} 
 if there is an orthogonal 
 matrix $U\in \QQ(\B)$, and potentially unitary is defined analogously using unitary matrices. 
The next result does not restrict matrices to being hollow.

\begin{thm}\label{t:bipart-unitary}
Let $G$ be a bipartite graph with no isolated vertices having partite sets of orders $r$ and $r'$.  Then $q(G)=2$ if and only if $r=r'$ and $\B_G$ is potentially orthogonal.
\end{thm}
\bpf It is immediate that  $\B_G$ is potentially orthogonal implies $q(G)=2$.  For the converse, we begin with the case in which $G$ is connected and $q(G)=2$.  Then $G$ has at least two vertices and and is potentially orthogonal by Remark \ref{r:bipart-q0-2}.  Let the two partite sets of vertices be $X = \{1, \ldots, r\}$ and $Y = \{r+1, \dots, n\}$. Then $|X| = r$ and $|Y| = r’:=n-r$.  Let $U\in\sym(G)$ be unitary.  Then $U=\mtx{D & B\\B^T & D'}$ with $D=\diag(d_1,\dots,d_r)$ and $D'=\diag(d'_1,\dots,d'_{r'})$, and
\[\mtx{I_r& O_{r,r'}\\O_{r,r'}& I_{r'}}=I = U^TU = U^2 = \mtx{D^2+BB^T & DB + BD'\\ B^TD + D'B^T & (D')^2+B^TB}.\]
Suppose that $b_{i,r+j}\ne 0$, i.e. $x_iy_j\in E(G)$.  Then
\[0=(U^2)_{i,r+j}=d_ib_{i,r+j}+b_{i,r+j}d'_j, \mbox{ so } d'_j=-d_i.\]
Since $G$ is connected, there is a path from any vertex to any other vertex.  Iterating $d'_j=-d_i$ for $x_iy_j\in E(G)$ shows that $D=dI_r$ and $D'=-dI_{r'}$ for some $d\in \R$.  Then $I_r=D^2+BB^T$ implies  $BB^T=(1-d^2)I_r$. Since $BB^T$ is positive semidefinite, $-1< d< 1$.  Also, $I_{r'}=(D')^2+B^TB$ implies   $B^TB=(1-d^2)I_{r'}$.   Thus $B'=\frac 1{\sqrt{1-d^2}}B\in\QQ(\B_G)$ is unitary, so $\B_G$ is potentially unitary.  Since both $BB^T$ and $B^TB$ are full rank, $r=r'$.

Finally, suppose $G=G_1\dcup\dots \dcup G_h$ where  $G_i$ are disjoint connected graphs each having order at least two for $i=1,\dots, h$.  Choose unitary $B_i\in\QQ(\B_{G_i})$ and define $U=B_1\oplus\dots\oplus B_h$.  Then $U$ is orthogonal, so $\B_G$ is potentially orthogonal.
\epf


Theorem \ref {t:bipart-unitary} would be false  if isolated vertices are allowed, as the next example shows.  Thus, the proof presented here corrects a minor error in Proposition 6.4 in \cite{AACFMN} (where the exclusion of isolated vertices was omitted).  

\begin{ex} Consider the graph $G=K_2\dcup K_1$. The matrix $A=\mtx{0 & 1 & 0\\1 & 0 & 0\\0 & 0 & 1}\in\sym(G)$, so $q(G)=2$.  However, $\B_G$ cannot be potentially orthogonal because the two partite sets cannot have equal size.
\end{ex}

\begin{cor}\label{t:bipart-unitary-q}Let $G$ be a  bipartite graph with no isolated vertices.  Then $q_0(G)=2$ if and only if $q(G)=2$. \end{cor}

The next result provides an example  of a bipartite graph $G$ having $q_0(G)> q(G)$. A \emph{spider} is a tree with exactly one vertex $v$ of degree 3 or more.  The spider $S(\ell_1,\dots,\ell_r)$ is the tree such that $S(\ell_1,\dots,\ell_r)-v=P_{\ell_1}\dcup\dots\dcup P_{\ell_r}$.

\begin{ex} Consider the graph $S(2,1,1)$. It was shown in \cite{IEPG2} that  $\oml(S(2,1,1))=\{(1,2,1,1),(1,1,2,1),$ $(1,1,1,1,1)\}$. Neither $(1,2,1,1)$ nor $(1,1,2,1)$ can be the ordered multiplicity list of a  matrix in $\symo(S(2,1,1))$ because neither allows a spectrum that is symmetric about the origin.  Thus $q_0(S(2,1,1))=5>4=q(S(2,1,1))$.
\end{ex}

Knowledge about potentially orthogonal patterns can sometimes be used with the bigraph pattern $\B_G$ of $G$ to construct a matrix $A\in\symo(G)$ with $q(A)=2$.   Sources of information about nonzero pattern matrices that allow orthogonality  include \cite{JMN08, diunitary, SS08} and the references therein.
Much of the literature concerns potentially unitary patterns rather than potentially orthogonal patterns, and it is known that there are patterns that  are potentially unitary but not potentially orthogonal \cite{hall}. However, few such examples are known, and many of the proofs of results that are stated for potentially unitary work for potentially orthogonal, as is the case with the next result. 
 A pattern $\B'=[\beta'_{ij}]$ is a {\em subpattern} of a pattern $\B=[\beta_{ij}]$ if $\beta'_{ij}=*$ implies $\beta_{ij}=*$, and in this case, $\B$ is a {\em superpattern} of  $\B'$.  The {\em $n\x n$ upper Hessenberg pattern} is 
 \[\HH_n=
 \mtx{* & * & \dots & * & * & *\\
 * & * & \dots & * & * & *\\
 0 & * & \dots & * & * & *\\
 \vdots & \vdots & \ddots & \vdots & \vdots\\
 0 & 0 & \dots & * & * & *\\
  0 & 0 & \dots & 0 & * & *}.\]
 \begin{thm}{\rm \cite{JMN08}} Any superpattern of the upper Hessenberg pattern $\HH_n$ is potentially orthogonal.
 \end{thm}


It should be noted that in the most of the work on potentially unitary patterns (including the papers cited), there is no assumption  that matrices be symmetric (even if pattern is symmetric).  Such results apply to demonstrating the existence of unitary matrices for bipartite graphs (using the bigraph pattern), but not to symmetric patterns that are not bipartite.


\section{Hollow inverse eigenvalue problem for specific families and small graphs}\label{s:fam-small}

We present the solution of the HIEP-$G$ for paths and complete bipartite graphs and  results on ancillary parameters such as $q_0$ for additional families of graphs in Section \ref{ss:fam}.  In  Section \ref{ss:small} we present results  for graphs of order at most 4.

\subsection{Families of graphs}\label{ss:fam}

\subsection*{Paths}


A stronger version of the HIEP-$G$ is solved for paths in \cite{AG15}, where a method is given to construct an $n\x n$ hollow matrix with any  set of $n$ distinct real numbers $\lambda_i$ that is symmetric about the origin as eigenvalues and additional numbers $\beta_i$ satisfying other conditions as the normalizing numbers (see \cite{AG15} for a definition of normalizing numbers); the solution to the  stronger version of the IEP-$G$ with normalizing numbers is also known \cite{AG15}.

    \begin{thm}\rm \cite{AG15} The multisets $S=\{\lambda_i:i=1,\dots, n\}$ and $T=\{\beta_i:i=1,\dots,n\}$ are the eigenvalues and normalizing numbers a of a symmetric tridiagonal matrix with positive sub- and superdiagonal entries
if and only if the following two conditions
are satisfied:
\ben[$(i)$] 
\item The numbers $\lambda_1,\dots,\lambda_n$ are real, distinct, and can be ordered so that $\lambda_1<\dots <\lambda_n$, $\lambda_k = -\lambda_{n+1-k}$ for $k = 1, 2,\dots , \lf \frac n 2\rf$, and $\lambda_{\lf \frac n 2\rf+1} = 0$ if $n$ is odd.
\item The numbers $\beta_1,\dots,\beta_n$ are positive, $\beta_1+\dots+\beta_n=1$ and can be ordered so that  $\beta_k = \beta_{n+1-k}$ for $k = 1, 2,\dots , \lf \frac n 2\rf$.
\een
\end{thm} 

\begin{cor}\label{c:path-HIEPG} A multiset $S$ of $n$ real numbers is the spectrum of some matrix $B\in\symo(P_n)$ if and only if  the entries of $S$ are distinct and $S=-S$.  The one and only ordered multiplicity list of the path is $(1,1,\dots,1)$, and it is bipartite spectrally arbitrary.
\end{cor}

\subsection*{Complete bipartite graphs}
We apply results in Section \ref{s:bipart} to solve the hollow IEPG for complete bipartite graphs. 
\begin{thm}\label{t:IEPG0-Kmn} Suppose $1\le m\le  n$.  A multiset $S$ of $m+n$ real numbers is the spectrum of some matrix $B\in\symo(K_{m,n})$ if and only if   $S=-S$ and the number of nonzero entries $k$ in $S$ satisfies $2\le k\le 2m$.  
\end{thm}
\bpf If $B\in\symo(K_{m,n})$, then $\spec(B)=-\spec(B)$ by Theorem \ref{t:bip-spec-sym}.  With the vertices appropriately ordered, $B=\mtx{O & A\\  A^T & O}$ where $A$ is an $m\x n$ matrix with every entry nonzero, so $2\le\rank B\le 2m$.

Now assume $S$ is a multiset  of $m+n$ real numbers satisfying $S=-S$ and the number of nonzero entries $k$ in $S$ satisfies $2\le k\le 2m$. Observe that $S=-S$ implies $k$ is even; let $\ell=\frac k 2$. Then  we can denote the entries of $S$ by $-\lambda_\ell\le \dots\le -\lambda_1< 0^{(n+m-2\ell)}< \lambda_1\le \dots\le \lambda_\ell$, with no assumption that the $\lambda_i$ are distinct.  Let $D=\diag({\lambda_1},\dots,{\lambda_\ell)}$ and  $\widehat D=\mtx{D & O_{\ell,n-\ell}\\O_{m-\ell,\ell} &O_{m-\ell,n-\ell}}\in \R^{m\x n}$. Choose an $n\x n$ real orthogonal   matrix $U$ and an $m\x m$ real orthogonal   matrix $W$ such that all entries of $A=W\widehat DU$ are nonzero (starting with random vectors and applying the Gram-Schmidt process to create orthonormal bases for $\R^n$ and $\R^m$  will accomplish this with high probability). 
  Then \[ \spec(AA^T)=\spec(W\widehat DUU^T\widehat D^TW^T)=\spec(\widehat D\widehat D^T)=\spec(D^2)\cup\{0^{(m-\ell)}\}=\{\lambda_1^2,\dots,\lambda_\ell^2\}\cup\{0^{(m-\ell)}\}\]
and \[\spec(A^TA)=\spec(AA^T)\cup\{0^{(n-m)}\}=\{\lambda_1^2,\dots,\lambda_\ell^2\}\cup\{0^{(n-\ell)}\}.\]
Define $B=\mtx{O_{m,m} & A\\  A^T & O_{n,n}}\in\symo(K_{m,n})$.  Then $B^2=\mtx{AA^T & O_{m,n}\\  O_{n,m}& A^TA}$, so $\spec(B^2)=\{(\lambda_1^2)^{(2)},\dots,(\lambda_\ell^2)^{(2)},$ $0^{(n+m-k)}\}$.  Since $\spec(B)=-\spec(B)$, $\spec(B)=S$.
\epf

\begin{cor}\label{q0Kmn}  For $n\ge 1$, $q_0(K_{n,n})=2$, and $q_0(K_{m,n})=3$ for $1\le m<n$.  \end{cor}

  The values $q(K_{n,n})=2$, and $q(K_{m,n})=3$ for $1\le m<n$ are established in \cite{AACFMN}, so $q_0(K_{m,n})=q(K_{m,n})$ for $1\le m\le n$.
Note that the IEPG for complete bipartite graphs has not yet been solved, and what is known is very different.  It is immediate from Theorem \ref{t:IEPG0-Kmn} that every hollow ordered multiplicity list that can be realized by $K_{m,n}$ is bipartite spectrally arbitrary.  However, for a matrix $A\in \sym(K_{m,n})$ with  ordered multiplicity list $(1,n+m-2,1)$ and $n\ge m\ge 3$, $\spec(A)=\{-\lambda+\mu, \mu^{(n+m-2)},\lambda+\mu\}$ for some $\lambda\in\R$  \cite{REU17}.  That is, the spectrum of $A$ is a translate of a  spectrum symmetric about the origin, so this  ordered multiplicity list  is not spectrally arbitrary for $K_{m,n}$.

\subsection*{Complete graphs and complete split graphs}

\begin{rem}\label{Kn-q0-mr0}
It was shown in Example \ref{q0-Kn} that $q_0(K_n)=2$ for $n\ge 2$, $\maxmulto(K_n)=n-1$, and $(n-1,1),(1,n-1)\in\omlo(K_n)$. 
\end{rem}

In contrast to the standard case, it is not the case that  every ordered multiplicity list with at least two entries can be realized by a matrix described by the complete graph. 

\begin{prop}\label{no22forK4} The ordered multiplicity list $(2,2)$  cannot be realized by any matrix in $\symo(K_4)$.  
\end{prop}
 \bpf  Suppose $(2,2)\in\omlo(K_4)$. 
 Then by Proposition \ref{oml-rr}, there is an orthogonal matrix $U=[u_{ij}]\in\symo(K_4)$, so  
 every off-diagonal entry of $U^2$ must be zero.  Solving $0=(U^2)_{1,2}=u_{1,3} u_{2,3} + u_{1,4} u_{2,4}$ gives $u_{2,4}= -\frac{u_{1,3} u_{2,3}}{u_{1,4}}$.
 Solving $0=(U^2)_{1,3}=u_{1,2} u_{2,3} + u_{1,4} u_{3,4}$ gives $u_{3,4}= -\frac{u_{1,2}u_{2,3}}{u_{1,4}}$.  Then $(U^2)_{1,4}=-\frac{2 u_{1,2} u_{1,3} u_{2,3}}{u_{1,4}}\ne 0$.
\epf



A \emph{complete split graph} is a graph of the form $\OL{K_{r}}\vee K_{s}$.  

\begin{thm}\label{JM12thm}{\rm \cite{JM12}} Let $n\ge\ell\ge 2$ and let $\lam_1\le \dots\le\lam_{n-\ell+1}<\lam_{n-\ell+2}=\dots=\lam_{n-1}=0<\lam_n$ be real numbers such that $\sum_{i=1}^n\lam_i=0$. Then there is a nonnegative matrix $A\in\symo(\OL{K_{\ell-1}}\vee K_{n-\ell+1})$ such that $\spec(A)=\{\lam_1,\dots,\lam_{n-\ell+1},0^{(\ell-2)},\lam_n\}$.
 If $\lam_1\le \dots\le\lam_{n-1}<0<\lam_n$ are real numbers such that $\sum_{i=1}^n\lam_i=0$, then there is a nonnegative matrix $A\in\symo(K_n)$ such that $\spec(A)=\{\lam_1,\dots,\lam_n\}$.
\end{thm}

Observe that $\OL{K_1}\vee K_{n-1}=K_n$.  As a result of Theorem \ref{JM12thm}, any spectrum with $n-1$ negative eigenvalues can be realized by some matrix in $\symo(K_n)$.  The next result converts a spectrum guaranteed by Theorem \ref{JM12thm} to an ordered multiplicity list (using $r=\ell-1)$.

\begin{cor}\label{JM12cor} For any $s$ positive integers $m_1,\dots,m_s$ such that $m_1+\dots+m_s=n-1$, \[(m_1,\dots,m_s,1),(1,m_s,\dots,m_1)\in\omlo(K_n).\] 
For any $s$ positive integers $m_1,\dots,m_s$ such that $m_1+\dots+m_s=n-r$, \[(m_1,\dots,m_s,r-1,1),(1,r-1,m_s,\dots,m_1))\in\omlo(\OL{K_{r}}\vee K_{n-r}).\]
\end{cor}

\begin{cor}\label{JM12cor2} For  $2\le r<n$,  $q_0(\OL{K_r}\vee K_{n-r})\le 3$.
\end{cor}

\begin{prop}\label{split2} For $n\ge 5$, $q_0(\OL{K_{n-2}}\vee K_2)=3$.
\end{prop}
\bpf Observe that the maximum order of a generalized cycle of $\OL{K_{n-2}}\vee K_2$ is 4, so $\MRo(\OL{K_{n-2}}\vee K_2)=4$ by Theorem \ref{maxrank}.  Since $4<n$, $q_0(\OL{K_{n-2}}\vee K_2)\ge 3$ by Proposition \ref{zero-eval}.  Corollary \ref{JM12cor2} completes the proof.
\epf

It was shown in \cite{mr0} that $\mro(K_n)=3$ for $n\ge 3$.  In particular, if  $n\ge 3$ and $T_n$ is the matrix whose $(i,j)$-entry is $(i-j)^2$, then $\rank T_n=3$.  Since $T_n$ is a primitive nonnegative matrix, $\rho(T_n)>|\lam|$ for every eigenvalue $\lam\ne \rho(T_n)$.  Thus $\spec(T_n)=\{\mu_1,\mu_2,0^{(n-3)},-\mu_1-\mu_2\}$.

\subsection*{Balanced Tripartite graphs}


\begin{prop} For $k\ge 1$, $(r,2r), (2r,r)\in\omlo(K_{r,r,r})$ and $q_0(K_{r,r,r})=2$.
\end{prop}
\bpf  Choose two $r\x r$ orthogonal matrices $V$ and $W$ such that there is no zero entry in any of $V$, $W$, and $VW^T$ (creating $V$ and $W$ by choosing random vectors and applying the Gram-Schmidt process will produce such matrices with high probability). Define a $3r\x 3r$ orthogonal matrix
\[U=\mtx{U_{11} & U_{12}\\U_{21} & U_{22}}\mbox{ where }U_{12}=\mtx{\frac 1{\sqrt 3} V\\ \frac 1{\sqrt 3}I_r}\mbox{ and } U_{22}=\frac 1{\sqrt 3}W. \] Define $M=U(I_{2r}\oplus -2I_r)U^T$.  Then
\[M=\mtx{U_{11} & U_{12}\\U_{21} & U_{22}}\mtx{I_{2r}&O_{2r,r}\\ O_{r,2r}&-2I_r }\mtx{U_{11}^T & U_{21}^T\\U_{12}^T & U_{22}^T}=
\mtx{U_{11}U_{11}^T-2U_{12}U_{12}^T&U_{11}U_{21}^T-2U_{12}U_{22}^T\\
U_{21}U_{11}^T-2U_{22}U_{12}^T&U_{21}U_{21}^T-2U_{22}U_{22}^T} .\]
Since $U$ is orthogonal:
\bit
\item
$U_{11}U_{11}^T-2U_{12}U_{12}^T=I_{2r}-3U_{12}U_{12}^T=\mtx{O_r& -V \\-V^T & O_r}$.
\item $U_{11}U_{21}^T-2U_{12}U_{22}^T=-3U_{12}U_{22}^T=\mtx{-VW^T \\-W^T}.$
\item $U_{21}U_{21}^T-2U_{22}U_{22}^T=I_r-3U_{22}U_{22}^T=O_r$.
\eit
Thus \[M=\mtx{O_r & -V & -VW^T\\-V^T & O_r & -W^T\\ -WV^T & -W & O_r}\in\symo(K_{r,r,r}).\]
Observe that $\spec(M)=\{(-2)^{(r)},1^{(2r)}\}$, so $(r,2r)\in\omlo(K_{r,r,r})$ and $2=q(A)=q_0(G)$.
\epf 


\subsection*{Hypercubes}
For the hypercube, $q_0(Q_d)=q(Q_d)=2$ by \cite{AACFMN} and Corollary \ref{c:qo2-cartprod}.  Furthermore, every ordered multiplicity list of length $2^{d-k}$ and of the form $(2^k,\dots,2^k)$ 
can be realized by a matrix in $\symo(Q_d)$ as follows: Apply Proposition \ref{p:allsimp} to choose a matrix in $ \symo(Q_{d-k})$ having all distinct nonzero eigenvalues.  Apply Lemma \ref{l:cartprod} to construct a matrix with ordered multiplicity list $(2^k,\dots,2^k)$.

\subsection*{Cycles}

The following information about minimum hollow rank is  known for cycles \cite{mr0}: If $n$ is even, then $\mro(C_n)=\mr(C_n)=n-2$, i.e., $\Mo(C_n)=\M(C_n)=2$. If $n$ is odd, then $\mro(C_n)=n$, i.e., $\Mo(C_n)=0$.  This implies $0$ cannot be an eigenvalue of $A\in\symo(C_n)$ for $n$ odd.
By Proposition \ref{p:allsimp}, $(1,1,\dots,1)\in\omlo(C_n)$ for $n\ge 3$.
The {\em flipped adjacency matrix} of $C_n$ is obtained from $\A(C_n)$ by replacing one symmetric pair of $1$s by $-1$s; denote the flipped adjacency matrix by $F_n$. It is known $\oml(F_n)=(2,2,\dots,2)$ if $n$ is even and  $\oml(F_n)=(1, 2,2,\dots,2)$ if $n$ is odd \cite{AIM}.  Thus $q_0(C_n)=q(C_n)=\lc\frac n 2\rc$ and $\maxmulto(C_n)=\M(C_n)=2$ for all $n\ge 3$.  Since $q_0(C_n)\ge 3$ for $n\ge 5$ and $0\not\in\spec(A)$ for $A\in\symo(C_n)$ and $n$ odd, no ordered multiplicity list is hollow spectrally arbitrary for odd $n\ge 5$. 


\subsection*{Wheels}


\begin{prop}\label{wheel}  For even $n\ge 4$,  $\Mo(W_n) = 1$. 
For odd $n\ge 7$, $\Mo(W_n) =3$. 
\end{prop}
\bpf   

We apply Lemma \ref{l:Zm-domvtx} to establish $\Mo(W_n) = 1$ for even $n\ge 4$ and $\Mo(W_n) = 3$ for odd $n\ge 7$ even.  Assume first   that $n\ge 4$ is even and let $A$ be the adjacency matrix of $C_{2k+1}$ where $2k+1=n-1$.  Then \[\spec(A)=\{2\}\cup\lsb\lp 2\cos\lp\frac{2j\pi}{2k+1}\rp\rp^{(2)}:j=1,\dots,k\rsb\] and $\bw=\lb 1,\cos(\frac{2\pi}{2k+1}),\dots,\cos(\frac{2(2k)\pi}{2k+1})\rb^T$ is an eigenvector for $\lam=2\cos(\frac{2\pi}{2k+1})$ that has every entry nonzero.  Define $B$ as in Lemma \ref{l:Zm-domvtx}.
 Then $n-1=\rank B\ge  \mro(W_n)\ge \mro(C_{n-1})=n-1$ (since $C_{n-1}$ is an odd cycle).
 
 Now assume   that $n\ge 7$ is odd and let $A$ be the adjacency matrix of $C_{2k}$ where $2k=n-1$.   Then $\spec(A)=\{2,-2\}\cup\lsb\lp 2\cos(\frac{2j\pi}{2k})\rp^{(2)}:j=1,\dots,k-1\rsb$ and $\bw=\lb 1,\cos(\frac{2\pi}{2k}),\dots,\cos(\frac{2(2k-1)\pi}{2k})\rb^T$ is an eigenvector for $\lam=2\cos(\frac{2\pi}{2k})$. Since $n\ge 7$, $\lam\ne 0$. If $k$ is odd, then every entry of $\bw$ is nonzero. If $k$ is even, then $\bw'=\lb 0,\sin(\frac{2\pi}{2k}),\dots,\sin(\frac{2(2k-1)\pi}{2k})\rb^T$  is also an eigenvector for $\lam=2\cos(\frac{2\pi}{2k})$, and for an appropriate choice of $b$, the linear combination $b \bw +\sqrt{1-b^2}\bw'$ has every entry nonzero.   Define $B$ as in Lemma \ref{l:Zm-domvtx}.
 Then $n-3=\rank B\ge  \mro(W_n)\ge \mro(C_{n-1})\ge \mr(C_{n-1})=n-3$.
\epf


\begin{prop}\label{W5-mro}\label{ex:AinW5}\label{q0-wheel}
$\Mo(W_5)=2$, $\maxmulto(W_5)=3$, and $q_0(W_5)=3$. 
\end{prop}
 \bpf Assume $5$ is the dominating vertex and number the remaining vertices in cycle order.
 
 We show first that $\mro(W_5)=3$, so $\Mo(W_5)=2$. Let $B\in\symo(W_{5})$.  For any nonsingular matrix $C$, $\rank(C^TBC)=\rank  B$.  By choosing a suitable diagonal matrix $C$, 
  \[ C^TBC=\mtx{0 & 1 &0 &a & b\\ 1 & 0 & 1 & 0 & c\\ 0 & 1 & 0 & 1 & d\\ a & 0 & 1 & 0 & e\\ b & c & d & e & 0}\]
and it is straightforward to verify that $\rank (C^TBC)\ge 3$.  Observe that the rank of the adjacency matrix of $W_{5}$ is $3$.  

Since it is known that $\M_+(W_5)=3$ \cite{IEPG2}, Lemma \ref{l:MMo-domvtx} implies that  $\maxmulto(W_5)\ge 3$.  Furthermore, $\maxmulto(W_5)\le \M(W_5)=3$.

 To see that $q_0(W_5)=3$, suppose to the contrary that we may find $A\in \symo(W_5)$ with $\oml(A)=(3,2)$.  By Proposition \ref{q02specarb}, we may assume $\spec(A)=\{-2^{(3)},3^{(2)}\}$.  Let $B=A(5)$, so $B\in\symo(C_4)$.  By interlacing, $\spec(B)=\{-2^{(2)},\lambda,3\}$, and by the trace constraint \eqref{trace-eq}, $\lambda=1$.  However, $\spec(B)$ is not symmetric about the origin, which is required by Theorem \ref{obs:bp_odd_zero} because $C_4$ is bipartite. Hence, $A$ cannot have $\oml(A)=(3,2)$ which implies $q_0(W_5) \geq 3$. Since $\maxmulto(W_5)=3$, $q_0(W_5) \leq 3$; therefore, $q_0(W_5) = 3$.
\epf

It is known that $q(W_5)=2$ since the IEP-$G$ is solved for all graphs of order at most 5 (see \cite{IEPG2}) so the previous result shows that $W_5$ is another example of a connected graph with $q_0(G)>q(G)$.

\begin{prop}\label{oml-wheel}  The set of hollow ordered multiplicity lists of $W_5$ is \[\omlo(W_5)=\{(3,1,1), (1,1,3),(2,1,2),(2,2,1),(1,2,2), (2,1,1,1),(1,1,1,2),(1,2,1,1), (1,1,2,1),(1,1,1,1,1)\}.\]
\end{prop}
\bpf 
Assume $5$ is the dominating vertex and number the remaining vertices in cycle order..
 
Suppose to the contrary that  $(1,3,1)\in\oml(W_5)$.  Then there exists a matrix $A\in\symo(W_5)$ with  $\spec(A)=\{\alpha,\beta^{(3)},\gamma\}$ where $\alpha < \beta < \gamma$.  Let $B=A(5)$.  Then $B\in\symo(C_4)$ and $\spec(B)=\{\lambda,\beta^{(2)},\mu\}$ where $\lam < \beta < \mu$ by interlacing and since $\maxmulto(C_4)=2$. Since $-\spec(B)=\spec(B)$, $\beta=0$ and $\lambda=-\mu$.  But $\beta=0$ implies $\rank A=2$, a contradiction to $\Mo(W_5)=2$.

  In all cases an ordered multiplicity list that can be realized can be reversed by considering the negative of of a realizing matrix. 
Lemma \ref{l:MMo-domvtx} implies that   there exists a  matrix $B\in\symo(W_5)$ such that $m_1(B)=\M_+(W_5)=3$.  Since $q_0(W_5)=3$, this implies that $(3,1,1)\in\omlo(W_5)$. Since $W_5$ has a $C_5$ as a subgraph, Proposition \ref{p:allsimp} implies $(1,1,1,1,1)\in\omlo(W_5)$.  We exhibit matrices for the remaining lists:\\
$A_{212}=\mtx{0 & 1 & 0 & -1 & 1 \\
1 & 0 & -1 & 0 & 1 \\
0 & -1 & 0 & 1 & 1 \\
-1 & 0 & 1 & 0 & 1 \\
1 & 1 & 1 & 1 & 0 }, \  A_{221}=\mtx{0 & 1 & 0 & 1 & \sqrt{2} \\
1 & 0 & 1 & 0 & \sqrt{2} \\
0 & 1 & 0 & 1 & \sqrt{2} \\
1 & 0 & 1 & 0 & \sqrt{2} \\
\sqrt{2} & \sqrt{2} & \sqrt{2} & \sqrt{2} & 0}\ A_{2111}=\mtx{0 & \frac{1}{\sqrt{6}} & 0 & \frac{1}{\sqrt{3}} & \sqrt{\frac{2}{3}} \\
 \frac{1}{\sqrt{6}} & 0 & \frac{1}{\sqrt{6}} & 0 & \frac{1}{6} \\
 0 & \frac{1}{\sqrt{6}} & 0 & \frac{1}{\sqrt{3}} & \frac{1}{\sqrt{6}} \\
 \frac{1}{\sqrt{3}} & 0 & \frac{1}{\sqrt{3}} & 0 & \frac{2 \sqrt{2}}{3} \\
 \sqrt{\frac{2}{3}} & \frac{1}{6} & \frac{1}{\sqrt{6}} & \frac{2 \sqrt{2}}{3} & 0}$\!, \ $A_{1121}=\A(W_5)$.
\bit
\item $\spec(A_{212})=\{(-2)^{(2)},0, 2^{(2)}\}$ and
$\oml(A_{212})=(2,1,2)$.
\item $\spec(A_{221})=\{(-2)^{(2)},0^{(2)},4\}$ and $\oml(A_{221}=(2,2,1)$.
\item  $\spec(A_{2111})=\{-1^{(2)},0,\frac{1}{2}
   \left(2-\sqrt{3}\right),\frac{1}{2} \left(2+\sqrt{3}\right)\}$ and
$\oml(A_{2111})=(2,1,1,1)$. 
\item $\spec(A_{1121})=\{-2,1-\sqrt{5},0^{(2)}, 1+\sqrt{5}\}$ and
$\oml(A_{1121})=(1,1,2,1)$. \qedhere
\eit
\epf

\subsection{Small graphs}\label{ss:small}

In this section we solve the HIEP-$G$ for connected graphs of order at most 3 and  determine the possible ordered multiplicity lists for all graphs of order $4$.   

\subsection*{Order $\le 3$}

\begin{obs}$\null$
\bit\item[ $K_1$:] For $A\in\symo(K_1)$, $\spec(A)=\{0\}$.


\item[$K_2$:] For $A\in\symo(K_2)$, $\spec(A)=\{-\lambda,\lambda\}$ and for any $\lambda\ne 0$ there is such an $A$.



\item[$P_3$:] Since $P_3=K_{1,2}$, $\spec(A)=\{-\lambda,0,\lambda\}$ for $A\in\symo(P_3)$ and for any $\lambda\ne 0$ there is such an $A$.

\eit
\end{obs}

\begin{prop}   Let  $\lambda_1<\lambda_2<0$ be real numbers. 
Then  for each of the multisets $\{\lambda_1^{(2)},-2\lambda_1\}$, $\{2\lambda_1,(-\lambda_1)^{(2)}\}$,  $\{\lam_1,\lam_2,-\lam_1-\lam_2\}$, and $\{\lam_1+\lam_2,-\lam_2,-\lam_1\}$, 
there is an   $A\in\symo(K_3)$ such that $A$ has the given spectrum.   Furthermore, $\spec(A)$ has one of these forms for every $A\in\symo(K_3)$.
\end{prop}

\bpf  Since $\mro(K_3)=3$ \cite{mr0}, zero cannot be an eigenvalue of any $A\in\symo(K_3)$.  Theorem \ref{JM12thm} shows that  the spectra $\{\lambda_1^{(2)},-2\lambda_1\}$ and $\{\lam_1,\lam_2,-\lam_1-\lam_2\}$ can be realized by  nonnegative matrices $A\in\symo(K_3)$, and the spectra $\{2\lambda_1,(-\lambda)^{(2)}\}$  and $\{\lam_1+\lam_2,-\lam_2,-\lam_1\}$ can be realized by the negatives of such  matrices.
These are the only possible forms that include at least two eigenvalues, do not have a zero eigenvalue, and satisfy the trace condition. 
\epf

\subsection*{Order 4}

There are 6 connected graphs of order 4: $P_4$, $K_{1,3}$, $C_4=K_{2,2}$, the paw graph (obtained by adding an edge to $K_{1,3}$), $K_4-e$ (also called the diamond graph), and $K_4$.  The HIEP-$G$ has been solved for $P_4$, $K_{1,3}$, $C_4=K_{2,2}$ (and is summarized in the next observation).  In this section we solve the HIEP-$G$ for disconnected graphs of order 4 and then determine the possible hollow ordered multiplicity lists for the remaining three connected graphs.
\begin{obs} $\null$
\bit
\item[$K_{1,3}$:]  $\spec(A)=\{-\lambda_1,0^{(2)},\lambda_1\}$ for $A\in\symo(K_{1,3})$, and for any $\lambda_1\ne 0$ there is such an $A$ by Theorem \ref{t:IEPG0-Kmn}.
\item[$C_{4}$:]   $\spec(A)=\{-\lambda_1,0^{(2)},\lambda_1\}$  or $\spec(A)=\{-\lambda_2,-\lambda_1,\lambda_1, \lambda_2\}$ for $A\in\symo(C_4)$, and for any $\lambda_1, \lambda_2\ne 0$ there is such an $A$ by Theorem \ref{t:IEPG0-Kmn} since $C_4=K_{2,2}$.
\item[$P_4$:]  $\spec(A)=\{-\lambda_2,-\lam_1,\lam_1,\lam_2\}$ with $0<\lam_1<\lam_2$ for $A\in\symo(P_4)$, and any such spectrum can be realized by some $A\in\symo(P_4)$ by  Corollary \ref{c:path-HIEPG}.
\eit
\end{obs}

As noted in Observation \ref{o:dunion}, spectra of disconnected graphs can be obtained from spectra of their connected components. 
\begin{rem} Observe that a component isomorphic to  $K_1$  will always contribute 0 to the spectrum.
\bit
\item If $A\in\symo(4K_1)$, then $\spec(A)=\{0^{(4)}\}$.
\item If $A\in\symo(K_2\dcup 2K_1)$, then $\spec(A)=\{-\lam_1,0^{(2)},\lam_1\}$ where $\lam_1> 0$, and for any $\lam_1> 0$ there is such an  $A\in\symo(K_2\dcup 2K_1)$.
\item For $A\in\spec(P_3\dcup K_1)$,  $\spec(A)=\{-\lambda_1,0^{(2)},\lambda_1\}$  where $\lam_1> 0$ and for any $\lambda_1> 0$ there is such an $A\in\symo(P_3\dcup K_1)$.
\item For $\lam_1<\lam_2<0$ and for each of the multisets $\{\lambda_1^{(2)},0,-2\lambda_1\}$, $\{2\lambda_1,0,(-\lambda_1)^{(2)}\}$,  $\{\lam_1,\lam_2,0,-\lam_1-\lam_2\}$, and $\{\lam_1+\lam_2,0,-\lam_2,-\lam_1\}$, there is an   $A\in\symo(K_3\dcup K_1)$ such that $A$ has the given spectrum. For $A\in\spec(K_3\dcup K_1)$, $\spec(A)$ has one of these forms.
\eit
\end{rem}

\begin{prop}\label{K4} The set of hollow ordered multiplicity lists of $K_4$ is 
$\omlo(K_4)=\{(3,1), (1,3), (2,1,1), (1,1,2), $ $(1,2,1), (1,1,1,1)\}$, and $\mro(K_4)=3$.
\end{prop}
 \bpf  By Proposition \ref{no22forK4},  $(2,2)\not\in\oml(K_4)$.
The ordered multiplicity lists $(3,1)$ and $(1,3)$ are realized by the adjacency matrix and its negative (see Example \ref{q0-Kn}).  The ordered multiplicity lists $(2,1,1), (1,1,2),$ and $(1,2,1)$ can be realized by Corollary \ref{JM12cor}. Since $\MRo(K_4)=4$, $(1,1,1,1)\in\omlo(K_4)$ by Proposition \ref{p:allsimp}.
It is well known  that $\mro(K_4)=3$ (see \cite{mr0}).
\epf

\begin{prop}\label{K4-e} 
The set of hollow ordered multiplicity lists of $K_4-e$ is $\omlo(K_4-e)=\{ (2,1,1), (1,1,2),$ $ (1,1,1,1)\} $, and $\mro(K_4-e)=3$.
\end{prop}
\bpf 
Since $(3,1),(1,3)\not\in\oml(K_4-e)$, $(3,1),(1,3)\not\in\omlo(K_4-e)$. Since $\mro(K_4-e)=3$ and  removing a degree 3 vertex from $K_4-e$ produces $P_3$, $(1,2,1)\not\in\omlo(K_4-e)$ by Proposition \ref{p:no121}. 

Suppose  $A\in\sym(K_4-e)$ and $(2,2)\in\oml(A)$. 
 Then by Proposition \ref{oml-rr}, there is an orthogonal matrix $U=[u_{ij}]\in\symo(K_4)$, so 
 every off-diagonal entry of $U^2$ must be zero.  
If edge $\{1,4\}$ is the edge that was deleted from $K_4$, then $(U^2)_{1,2}=u_{1,3}u_{3,2}\ne 0$.  Thus $(2,2)\not\in\oml(A)$.

  Since $K_4-e=\OL{K_2}\vee K_2$, the ordered multiplicity lists $(2,1,1)$ and $ (1,1,2)$  can be realized by Corollary \ref{JM12cor}. Since $\MRo(K_4-e)=4$, $(1,1,1,1)\in\omlo(K_4-e)$ by Proposition \ref{p:allsimp}. 

Since $K_4-e$ has an induced $C_3$, $\mro(K_4)\ge 3$.  Since $K_4-e$ does not have a unique generalized cycle of order $4$, $\mro(K_4-e)\le 3$. \epf

\begin{prop}\label{paw} Let $G$ be the paw graph.  Then the set of hollow ordered multiplicity lists of $G$ is 
$\omlo(G)= \{(2,1,1), (1,1,2),$ $ (1,1,1,1)\}$, and $\mro(G)=4$. 
\end{prop}
\bpf  
Since $(3,1),(1,3), (2,2)\not\in\oml(G)$ \cite{IEPG2}, $(3,1),(1,3),(2,2)\not\in\omlo(G)$.  Since  removing a degree 2 vertex from the paw produces $P_3$, $(1,2,1)\not\in\omlo(G)$ by Proposition \ref{p:no121}. 

  The ordered multiplicity lists $(2,1,1)$ and $ (1,1,2)$  can be realized by Lemma \ref{r:m1}. Since $\MRo(G)=4$, $(1,1,1,1)\in\omlo(G)$ by Proposition \ref{p:allsimp}. 

Note that $\mro(G)=4$ since the paw has a unique generalized cycle of order 4 consisting of two disjoint edges. 
\epf


\section*{Acknowledgements}

This research was partially supported by NSF grant 1916439.

\end{document}